\documentclass[12pt, a4paper,reqno]{amsart}
\usepackage{amsmath,amsfonts,amssymb}

\theoremstyle{plain}
\usepackage{color}
\newtheorem{theorem}{Theorem}

\newtheorem{proposition}{Proposition}[section]
\theoremstyle{proof}
\theoremstyle{definition}

\theoremstyle{remark}

\theoremstyle{lamma}

\numberwithin{equation}{section}
\numberwithin{lemma}{section}
\numberwithin{theorem}{section}

\usepackage{amsmath}
\usepackage{amsfonts}   
\usepackage{amssymb}
\usepackage{amssymb, amsmath, amsthm}
\usepackage[breaklinks]{hyperref}
\theoremstyle{thmrm}

\usepackage{graphicx}
\begin{document} 
\title[Lebesgue-Ramanujan-Nagell type  equations]{Complete solutions of  certain Lebesgue-Ramanujan-Nagell type  equations}
\author{Kalyan Chakraborty, Azizul Hoque and Richa Sharma}
\address{Kalyan Chakraborty @Harish-Chandra Research Institute,
Chhatnag Road, Jhunsi,  Allahabad 211 019, India.}
\email{kalyan@hri.res.in}
\address{Azizul Hoque @Harish-Chandra Research Institute,
Chhatnag Road, Jhunsi,  Allahabad 211 019, India.}
\email{ahoque.ms@gmail.com}
\address{Richa Sharma @Malaviya National Institute of Technology Jaipur,  Jaipur, 302017, India.}
\email{richasharma582@gmail.com}

\maketitle

\begin{abstract}
It is well-known that for $p=1, 2, 3, 7, 11, 19, 43, 67, 163$, the class number of $\mathbb{Q}(\sqrt{-p})$ is one. We use this fact to determine all the solutions of $x^2+p^m=4y^n$ in non-negative integers $x, y, m$ and $n$.   
\end{abstract}

\section{Introduction}
Finding the solutions of an exponential Diophantine equation is a classical problem and even now it is a very active area of research. One of the  most prominent equation of this type is the generalized Lebesgue-Ramanujan-Nagell equation: 
\begin{equation}\label{LRNE}
x^2+D^m=\lambda y^n,
\end{equation}
where $D$ and $\lambda$ are fixed positive integers. Here, one of the problems is to find all the solutions of \eqref{LRNE} in positive integers $x, y, m $ and $n$. Many special cases of \eqref{LRNE} have been considered over the years, but almost all the results for general values of $m$ and $n$ are of fairly recent origin. 

When $ y=C$ is also fixed in \eqref{LRNE}, the resulting equation, 
\begin{equation}\label{RNE}
x^2+D^m=\lambda C^n,
\end{equation}
is called  the generalized Ramanujan-Nagell equation.
The famous Ramanujan-Nagell equation,
\begin{equation}\label{RE}
x^2+7= 2^n,
\end{equation}
 is a particular case of \eqref{RNE}. S. Ramanujan \cite{RA13} in 1913, conjectured that the complete set of solutions
of \eqref{RE} is given by 
$$(x,n)=\{(1,3), (3,4), (5,5), (11,7),(181,15)\}.$$
It was T. Nagell, who  confirmed this conjecture in 1948, and his proof in english was published in 1960 (see, \cite{NA61}). 
We refer \cite{BP14, SS06}  for further results on the solutions  of \eqref{RNE}.

A  Lebesgue-Nagell type equation, 
\begin{equation}\label{LNE}
x^2+D=\lambda y^n, 
\end{equation} 
is a particular case of \eqref{LRNE}. 

There are many interesting results concerning the solutions of \eqref{LNE} for $\lambda=1$. For $\lambda=1$, this equation 
has been extensively studied by many authors (see \cite{BS04, BM20, JC93, JC03, HS16}), and thus there are many interesting results concerning its solutions. 
 The equations of the form \eqref{LRNE} are well connected with the investigation of the class number of imaginary quadratic field $\mathbb{Q}(\sqrt{-D})$. The authors used the solvability of some special cases of \eqref{LRNE}  to study the class number of certain imaginary quadratic  fields in \cite{CHKP, CH, CHS}. S. A. Arif and F. S. A. Muriefah \cite{AM02} investigated \eqref{LRNE} for the case $\lambda=1, ~n\geq 5,$ $m$ odd integers and $D=p$ an odd prime. Precisely, they proved that the resulting  equation has no solution $(x, y, m, n)$ under the conditions $p\not\equiv 7\pmod8, ~ 3\nmid n$ and $n$ is co-prime to the class number of $\mathbb{Q}(\sqrt{-p})$ except for $p=19, 341$. For these values of $p$, the equation has exactly two families of solutions under the same conditions. A. B\'{e}rczes and I. Pink \cite{BP08} considered \eqref{LRNE} under the conditions $\lambda=1, ~ 2\mid m$, $D=p\leq 100, ~n$ is an odd primes and $\gcd(x,y)=1$. In \cite{SS06, HM11}, the authors extended the results of \cite{AM02, BP08} for some composite values of $D$. On the other hand, A. B\'{e}rczes and I. Pink \cite{BP} further extended the result of \cite{SS06} to the case where the class number of $\mathbb{Q}(\sqrt{-d})$ is supposed to be $2$ or $3$.

 For $\lambda=2$, Sz. Tengely \cite{TE} considered \eqref{LRNE} when $D=p, ~m=2q$, where both $p>2$ and $q>3$ are primes. In this case, he proved that it has finitely many solutions $(x, y, p, q)$ provided $y$ can not be written as the sum of two consecutive squares. A more general version of this case was studied by F. S. A. Muriefah, F. Luca, S. Siksek and Sz. Tengely in \cite{MFST} when $D \equiv 1 \pmod 4$. On the other hand for the case $(m,\lambda)=(1,4)$, F. Luca, Sz. Tengely and A. Togb\'{e} \cite{LTT} studied the solutions $(x, y, n)$ of \eqref{LRNE} when $D\equiv 3\pmod 4$ with $D\leq 100$.

In this paper, we  find all the solutions of  \eqref{LRNE} in non-negative integers $x, y,m$ and $n$ for  $\lambda =4$ and $D=\{1, 2,3, 7, 11, 19, 43, 67, 163\}$. More precisely, we find all the solutions of the equation,
\begin{equation}\label{me}
x^2+p^m=4y^n
\end{equation}
in non-negative integers $x, y, m$ and $n$ for $p\in\{1, 2, 3, 7, 11,19, 43, 67,163\}$. 
We note that \eqref{me} has been solved completely for $p=19$ in \cite{BHS18}. It is easy to see that \eqref{me} has no solution for $p=1$.  For $p=2$, $x$ is even and thus by writing $x=2X$, we get (from \eqref{me}) 
$$X^2+2^{m-2}=y^n.$$
One can explicitly write down the solutions of this equation using \cite[pp. 69-70]{AM01} and \cite[Theorem]{CO92}. Thus we exclude the cases when $p=1, 2, 19$. 

\section{Statement of the result}
We first note that to find all the solutions $(x, y, m, n)$ of \eqref{me}, it is sufficient to consider $x, y\geq 2$ and $m, n\geq 1$, since one can easily check it for smaller non-negative values of $x, y, m$ and $n$. Also, \eqref{me} has no solution when $m $ is even. Therefore the problem of finding all the solutions $(x, y, m,n)$ of \eqref{me} reduces to finding all the solutions $(x, y, k, n)$ of the following:
\begin{equation}\label{me2}
x^2+p^{2k+1}=4y^n,
\end{equation}
where $p=3, 7, 11, 43, 67, 163$.

When $n=1$, it is easy to see that  \eqref{me2}  has infinitely many solutions, and all of them are given by the following parametric form: 
\begin{eqnarray*}
\begin{cases}
x=2t+1, \\ y=t^2+t+\frac{1+p^{2k+1}}{4}, \text{ where}~ \ t \in  \mathbb{Z}_{\geq 0}.
\end{cases}
\end{eqnarray*}
We are now in a position to state the main result concerning the solutions of \eqref{me2} in positive integers $x,y,k$ and $n$.

\begin{theorem}\label{thm} 
Let $k\geq 0$ and $n\geq 2$ be two integers. Then \eqref{me2} has no solutions $(x, y, k, n)$ except for $n \in \{2, 3, 5, 7, 13\}$ and $3\mid n$. 
For these values of $n$, the 
families of solutions are given by Table \ref{T1} (with $t,r,s\in\mathbb{Z}_{\geq 0})$. 
\begin{table}[ht]
 \centering
\begin{tabular}{ c c  c  c  c} 

 $x$ & $y$ & $p$&$k$ & $n$\\
\hline
\vspace*{4mm}
$p^t\times \frac{p^{2(k-t)+1}-1}{2}$&$p^t\times\frac{p^{2(k-t)+1}+1}{4}$& $p$ &$k$ &$2$\\ 
$37\times 3^{3r}$&$7\times 2^{2r}$&$3$ & $3r$& $3$\\ \vspace*{2mm}
$5\times 7^{3r}$& $2\times 7^{2r}$& $7$& $3r$ & $3$\\ \vspace*{2mm}
$11\times 7^{5r}$& $2\times 7^{2r}$& $7$& $5r$ & $5$\\ \vspace*{2mm}
$31\times 11^{5r}$& $2\times 11^{2r}$& $11$& $5r$ & $5$\\ \vspace*{2mm}
$13\times 7^{7r}$& $2\times 7^{2r}$& $7$& $7r+1$ & $7$\\ \vspace*{2mm}
$181\times 13^{13r}$& $2\times 13^{2r}$& $7$& $13r$ & $13$\\ \vspace*{2mm}
$3^{r+1}$& $3^{(2r+1)/3s}$ & $3$ & $r$ & $3s$\\
\hline
\end{tabular}
\vspace*{1mm}
\caption{\small Solutions of \eqref{me2}}\label{T1}
\end{table}
\end{theorem}
We use elementary arguments and some properties of Lucas numbers to prove Theorem \ref{thm}.  We give the definitions of Lucas numbers and its primitive divisors for the sake of completion. Let $\alpha$ and $\beta$ be two algebraic numbers satisfying:
\begin{itemize}
\item $\alpha +\beta$ and $\alpha \beta$ are non-zero co-prime rational integers,
\item $\frac{\alpha}{\beta}$ is not a root of unity.
\end{itemize}
 Then the sequence $(u_{n})_{n=0}^{\infty}$ defined by
 $$
 u_{n} = \frac{\alpha^{n}- \beta^{n}}{\alpha-\beta}, \quad n\geq 1,
 $$ 
is a Lucas sequence. 
We say that a prime number $p$ is a primitive divisor of a Lucas number $u_{n}$ if $p$ divides $u_{n}$, but does not divide $(\alpha-\beta)^{2}u_{2} \cdots u_{n-1}.$ 
  
\section{Proof of Theorem \ref{thm}}
It is easy to see from \eqref{me2} that both $x$ and $y$ are odd except when $p=7$. When $p=7$, we see that $x$ is odd and $y$ is even. We first prove the following crucial proposition.

\begin{proposition}\label{prop3.1}
Let $k$ and $n$ be as in Theorem \ref{thm}. Then all the solutions $(x,y,p,k, n)$ of \eqref{me2} with $\gcd(p,x)=1$ and odd $n$ are given by Table \ref{TP}.
\begin{table}[ht]
 \centering
\begin{tabular}{ c c  c  c  c} 

 $x$ & $y$ & $p$&$k$ & $n$\\
\hline

$37$&$7$&$3$ & $0$& $3$\\ 
$5$& $2$& $7$& $0$ & $3$\\ 
$11$& $2$& $7$& $0$ & $5$\\ 
$31$ & $3$ & $11$& $0$ & $5$\\ 
$13$& $2$& $7$& $1$ & $7$\\ 
$181$& $2$& $7$& $0$ & $13$\\
\hline
\end{tabular}
\vspace*{1mm}
\caption{\small Solutions of \eqref{me2} when $\gcd(p,x)=1$ and $n$ is odd}\label{TP}
\end{table}

\end{proposition}
\begin{proof}
It is sufficient to find the solutions of \eqref{me2} for all odd prime values  $q$ of $n$, that is, to find all the solutions $(x, y, k, q)$ of  
\begin{equation}\label{p3}
x^2+p^{2k+1}=4y^q,
\end{equation}
where $p=3, 7, 11, 43, 67$ and $163$.

We can factorize $y^q$ uniquely as the class number of $\mathbb{Q}(\sqrt{-p})$ is one:
\begin{equation*}
 \left( \frac{x+p^{k} \sqrt{-p}}{2}\right)  \left( \frac{x-p^{k} \sqrt{-p}}{2}\right) =y^q.
\end{equation*}
Thus for some rational integers $a$ and $b$ with same parity, we have
\begin{equation} \label{eq1f}
\frac{x+p^{k} \sqrt{-p}}{2} =u\left( \frac{a+b \sqrt{-p}}{2}\right)^q
\end{equation}
satisfying $y= \frac{a^{2}+pb^{2}}{4}$ with $u$ is a unit in the ring of integers of $\mathbb{Q}(\sqrt{-p})$. Now the only units in the ring of integers of $\mathbb{Q}(\sqrt{-p})$ are $\pm 1$ when $p\ne 3$. These can be absorbed into the $q$-th power. For $p=3$, the units are $\pm 1, \pm\omega, \pm \omega^2$, where $\omega=\exp(2\pi i/3)$, and all of them satisfy $u^6=1$. Thus these units can also be absorbed into the $q$-th power except when $q=3$. Therefore \eqref{eq1f} implies
\begin{equation}\label{p5}
\frac{x+p^{k} \sqrt{-p}}{2} =\left( \frac{a+b \sqrt{-p}}{2}\right)^q
\end{equation}
 satisfying $y= \frac{a^{2}+pb^{2}}{4}$ except in the case when $p=q=3$.
Since $x$ odd, it is easy to see that both $a$ and $b$ are  also odd. 

We fix $ \alpha =\frac{a+ b \sqrt{-p}}{2}$ and 
 \begin{align*} \label{4}
u_i=\frac{\alpha^{i}-\bar{\alpha}^{i}}{\alpha-\bar{\alpha}}~~( i=0, 1, 2, 3, \cdots),
\end{align*}
where $\bar{\alpha}$ denotes the conjugate of $\alpha$. Here both $\alpha$ and $\bar{\alpha}$ are algebraic integers as well as $\gcd( \alpha + \bar{\alpha}, \alpha \bar{\alpha})=1$. We observe that $\frac{\alpha}{\bar{\alpha}}$ is not a root of unity in the ring of integers of $\mathbb{Q} (\sqrt{-p}) $. Thus $(\alpha , \bar{\alpha})$ is a Lucas pair, and hence
$u_i=\frac{\alpha^{i}-\bar{\alpha}^{i}}{\alpha-\bar{\alpha}}$ is a Lucas number. Now it follows from \eqref{p5} that $\alpha^q-\bar{\alpha}^q=p^k\sqrt{-p}$. Thus $b$ divides $p^k$ and $u_q=p^k/b$ has no prime factor other than $p$. Hence, $u_q$ has no primitive divisor since $p$ divides $(\alpha-\bar{\alpha})^2$. Using a result of Bilu, Hanrot and Voutier \cite{BH01}, one concludes that \eqref{p3} has no solution for all primes $q>13$.

On the other hand, if $ q \in  \{3,5,7,11,13 \}$ then there are Lucas pairs $(\alpha,\bar{\alpha})$ for which $u_q$ does not have primitive divisors. We consider each of these primes separately.

For $q=13$, the only Lucas pair without primitive divisors belongs to $\mathbb{Q} (\sqrt{-7}) $ and the corresponding $\alpha=\frac{1+\sqrt{-7}}{2}$. Now comparing with our fixed $\alpha$, we get $(a,b)=(1,1)$ and hence, $y=2$. Thus \eqref{p3} implies $$x^2+7^{2k+1}=4\times 2^{13}$$ which further implies that $(x,k)=(181,0)$. Therefore $(x, y, k, q)=(181, 2,0,13)$ is a solution of \eqref{p3}. 

When $ q=11$, there are no Lucas pairs without primitive divisors, and hence in this situation we do not get any solution. 

Again for $q=7$, the only Lucas pair without primitive divisors belongs to $\mathbb{Q} (\sqrt{-7}) $ which corresponds to $\alpha=\frac{1+7\sqrt{-7}}{2}$ and thus we get $(a,b)=(1,1)$. This gives $y=2$ and hence we have 
$$
x^{2}+7^{2k+1}=4\times 2^{7}. 
$$
It is easy to check that $(x, k)=(13, 1)$ is the only solution of this equation. Therefore the corresponding solution of \eqref{p3} is $(x, y, k, q)=(13, 2, 1, 7)$.  

Further for $q=5$, the only two Lucas pairs without primitive divisors belong to $\mathbb{Q} (\sqrt{-7}) $ and  $\mathbb{Q} (\sqrt{-11}) $, and are given by $\alpha=\frac{1+\sqrt{-7}}{2}$ and $\alpha=\frac{1+\sqrt{-11}}{2}$ respectively. This further implies that $(a,b)=(1,1)$ and hence $y=2,3$ respectively.  Thus \eqref{p3} reduces to
$$
x^{2}+7^{2k+1}=4\times 2^5, \text{ or } x^2+11^{2k+1}=4\times 3^5.
$$
We see that $(x,k)\in\{(11,0), (31,0)\}$ are the only solutions of these equations. Thus the corresponding solutions of \eqref{p3} are $$(x, y, k,q)\in\{(11,2,0,5), (31, 3, 0, 5)\}.$$ 
We finally consider the case $q=3$.

(i) $k=0$.  In this case \eqref{p3} becomes, 
\begin{equation*}\label{ex}
x^2+p=4y^3.
\end{equation*}
By \cite[Theorem 1.1]{LTT}, we see that it  has only two solutions, viz. $(x, y, p)\in\{(37, 7, 3), (5, 2, 7)\}$ except for $p=163$. The corresponding solutions of \eqref{p3} are $(x, y, k,q)\in\{(37, 7,0,3), (5, 2, 0, 3)\}$.

We now deal with the case when $k=0$ and $p=163$;  equating the imaginary parts of \eqref{p5} we get,
$$
4=b(3a^2-163b^2).
$$
As $b$ is odd, the only possible values of $b$ are $\pm 1$, and thus
$\pm 4=3a^2-163$. 
This is not possible since $a$ is an odd integer. 

(ii)  $k\geq 1$. Again  equating the real and imaginary parts of \eqref{p5} except  when $p=3$, we get
\begin{align} \label{2a}
4x=a^{3}-3ab^{2}p
\end{align}
 and 
\begin{align} \label{2b}
4p^{k}=3a^{2}b-pb^{3}.
\end{align}
The possible values of $b$ are $\pm p^t,~ 0\leq t\leq k$ (as $b$ is odd).

For $t=0$; \eqref{2b} reduces to
$$
\pm 4p^k=3a^2-p.
$$
As $p\ne 3$ and $k\geq 1$, one derives from above relation that $p\mid a$, which is a contradiction.

For $1\leq t\leq k$; \eqref{2b} gives
$$
\pm 4 p^{k-t}=3a^{2}-p^{2t+1}.
$$
Since $\gcd(a, p)=1$ and $p\neq 3$, one derives that  $k=t$, and thus the above equation reduces to
\begin{align} \label{2c}
\pm 4=3a^{2}-p^{2k+1}.
\end{align}
As $p\equiv 1\pmod 3$, except for $p=11$,  reading \eqref{2c} modulo $3$ we can avoid positive sign from the l.h.s. of \eqref{2c}. Thus \eqref{2c} reduces to
\begin{align} \label{p17}
-4=3a^{2}-p^{2k+1}.
\end{align}
Similarly for $p=11$, by reading \eqref{2c} modulo $3$ we can avoid negative sign from the l.h.s. of \eqref{2c}, and thus it reduces to
\begin{align} \label{p18}
4=3a^{2}-p^{2k+1}.
\end{align}
For $p=163$, reading \eqref{p17} modulo $9$ we obtain $a^2\equiv 2\pmod 3$. This is not possible.  

Again for $p=67$, reading \eqref{p17} modulo $11$ we obtain $a^2\equiv 10\pmod {11}$ which is not possible.   

When $p=43$, reading \eqref{p17} modulo $11$ we obtain $a^2\equiv 2\pmod {11}$ which is again not possible.

Let us now  consider $p=11$. In this case reading \eqref{p18} modulo $8$ we obtain $a^2\equiv 5\pmod {8}$, which is impossible.

For $p=7$,  \eqref{p17} reduces to
\begin{equation}\label{p19}
3a^2+4=p^{2k+1}.
\end{equation}
Using \cite[Theorem 1]{BS01}, we can conclude that \eqref{p19} has at most one solution $(a,k)$ and thus $(a, k)=(1, 0)$ is the only solution of \eqref{p19}. Therefore \eqref{2a} gives $x=-5$ as $b=\pm 7^k$ with $k=0$. This is not possible since $x$ is positive.  

We finally consider the remaining case $p=q=3$ with $k\geq 1$. The units in ring of integers of $\mathbb{Q} (\sqrt{-3}) $ are $\pm1, \pm \omega, \pm \omega^{2}$, where $\omega=\exp(\frac{2 \pi i}{3})$.
 Thus \eqref{eq1f} reduces
\begin{eqnarray}\label{pex}
\frac{x+3^k\sqrt{-3}}{2} =\begin{cases}\left( \frac{1\pm \sqrt{-3}}{2}\right)\left( \frac{a+b \sqrt{-3}}{2}\right) ^{3}, \text{ or}\\ 
\left( \frac{a+b \sqrt{-3}}{2}\right) ^{3}
\end{cases}
\end{eqnarray}
with $a\equiv b \equiv 1 \pmod 2$ satisfying $4y= a^{2}+3b^{2}$.

We equate the imaginary parts of the first case of \eqref{pex} to get
$$
8\times 3^{k}=3a^{2}b-3b^3\pm a^{3}\mp 9ab^{2}.
$$
This shows that $3\mid a$ since $k\geq 1$, which is a contradiction.
Similarly, equating the imaginary parts of the case of \eqref{pex} we get $4\times 3^k=3a^2b-3b^3$. This implies that $b=3^r$ since $b$ is odd and thus one gets $4\times 3^{k-r-1}=a^2-3^{2r}$ for some integer $r\geq 0$. This is not possible since $4\times 3^{k-r-1}\equiv 0\pmod 8$. 
\end{proof}

\subsection*{Proof of Theorem \ref{thm}}
We present the proof in two parts: $ \gcd(p,x)=1$ and $ \gcd(p,x)\neq 1$.
\\
Case I: $\gcd(p,x)=1$.\\
In this case, Proposition \ref{prop3.1} gives all the solutions of \eqref{me2} provided $n>1$ is odd.
We now consider $n$ to be even and thus \eqref{me2} can be written as  
\begin{equation}\label{p1}
(2y^t-x)(2y^t+x)=p^{2k+1}
\end{equation}
with $n=2t$ for some integer $t\geq 1$.
Now suppose that $\gcd (2y^t-x,2y^t+x)\neq 1$. Then by \eqref{p1},  we obtain that $p\mid \gcd(2y^t-x, 2y^t+x)$ and that gives $p\mid x$, which is a contradiction. Thus $\gcd (2y^t-x,2y^t+x)= 1$, and hence  \eqref{p1} gives
\begin{eqnarray}\label{p2}
\begin{cases}
2y^{t}-x  = 1\\
2y^{t}+x =p^{2k+1}.
\end{cases}
\end{eqnarray}
Adding this two equations and then reducing modulo $3$ one gets,
$$
y^{t}\equiv 2 \pmod 3,
$$
except for $p=3$ and $11$. This shows that $t$ is odd and $y\equiv 2 \pmod 3.$ Similarly for $p=11$, once again adding the equations in \eqref{p2} and further reducing modulo
$5$ we obtain
$$
y^t\equiv 3 \pmod 5.
$$
Since $3$ is a quadratic non-residue modulo $5$, one concludes that $t$  is odd. Therefore \eqref{p1} reduces to 
$$
x^{2}+p^{2k+1}=4Y^t,
$$
where $Y=y^2$ and $t$ is an odd integer. This has no solution $(x,Y,p,k,t)$ by Proposition \ref{prop3.1}  
except for $t=1$. 

When $p=3$, equation \eqref{p1} becomes
$$
x^2+3^{2k+1}=4y^{2^{r}t}
$$
for some odd integer $t$ and for some integer $r$. This can further be simplified to
$$
x^2+3^{2k+1}=4Y^t,
$$
where $Y=y^{2^r}$. This has no solution $(x,Y,k,t)$ by Proposition \ref{prop3.1}  
except for $t=1$. Assume that  $r>1$. Then by \eqref{p2} we have a solution provided $(3^{2k+1}+1)/4=N$ for some integer $N\geq1$. This implies $k=0$ and hence, $x^2+3=4y^{2^r}$. By \cite[Theorem 1.1]{LTT}, we see that $(x,y)=(1,1)$ is the only solution which is already considered.

For $t=1$, \eqref{p2} gives 
$$(x,y)=\left(\frac{p^{2k+1}-1}{2}, \frac{p^{2k+1}+1}{4}\right).$$
This implies that
$$(x,y,k,n)=\left(\frac{p^{2k+1}-1}{2}, \frac{p^{2k+1}+1}{4}, k, 2\right)$$
is a family of solutions of \eqref{me2}. 
\\
Case 2: $\gcd(p,x)\neq 1$. \\ 
We put $x=p^sX$, and $y=p^tY$ for some positive integers $s$ and $t$ satisfying $\gcd(p,X)=\gcd(p,Y)=1$. Then \eqref{me2} turns out to
 \begin{equation}\label{p20}
  p^{2s}X^{2}+p^{2k+1}=4p^{tn}Y^n.
 \end{equation}
We encounter three possibilities here and the first one is 
$2k+1= \mbox{min} \{2s,2k+1,t n\}$.  We note that we can not have $n=2$ in this case since $4p^{2t}Y^2-p^{2s}X^2=p^{2k+1}$ would imply $2k+1=\min\{2s, 2t\}$, a contradiction. Therefore \eqref{p20} reduces to
 $$
  p(p^{s-k-1}X)^{2}+1=4Y^{n}p^{tn-2k-1}, ~~ n>2.
 $$
Utilising previous technique we read it modulo $p$, and conclude that $t n = 2 k +1$. Thus the last equation becomes
 $$
 p(p^{s-k-1}X)^{2}+1=4Y^{n}, ~~ n>2.
 $$
This can be written as 
\begin{equation}\label{ez}
pZ^2+1=4Y^n, ~~Z=p^{s-k-1}X,~~n>2.
\end{equation}
Let us assume $3$ does not divide $n$. Then as in the proof of Proposition \ref{prop3.1}, we put $Y=(a^2+pb^2)/4$ and $\alpha=(a+b\sqrt{-p})/2$ to obtain 
$$\frac{1+Z\sqrt{-p}}{2}=\alpha^n.$$
Now $u_j=\frac{\alpha^j-\bar{\alpha}^j}{\alpha-\bar{\alpha}},~ j=0,1,2,\cdots$, is a Lucas sequence and $u_{2n}/u_n=\alpha^n+\bar{\alpha}^n=\pm 1$. Here, we are looking for those integers $n$ for which $u_{2n}$ has no primitive divisor. Thus, from \cite[Table 1]{BH01} and the assumption $3\nmid n$,
we must have $(a, b, p, n)=(1, 1, 7, 4), (1, 1, 11, 5)$ or $(5, 1, 3, 5)$. After examining each case, we see that $(a,b,p,n)=(1,1,7,4)$ is the only possibility. For this possible value, \eqref{ez} gives $(X,Y,p,n)=(3,2,7,4)$, which does not contribute any solution to \eqref{me2}, since $nt=2k+1$.
We now consider that $n$ is a multiple of $3$. Then \eqref{ez} can be written as $pZ^2+1=4W^3$ with $W=Y^{n/3}$. If $p\ne 3$, then one gets $$\frac{1+Z\sqrt{-p}}{2}=\left(\frac{a+b\sqrt{-p}}{2}\right)^3.$$ 
Equating the real parts, we get $4=a(a^2-3b^2p)$, which is impossible. On the other hand, if $p=3$ then one gets $(1+Z\sqrt{-p})/2=(a+b\sqrt{-p})/2)^3$ or $(1+Z\sqrt{-p})/2=((1\pm \sqrt{-3})/2)((a+b\sqrt{-p})/2)^3$. The former case is impossible as above. In the latter case, one gets $8=a^3-9ab^2\pm 9a^2b\mp 9b^3$. This implies $(a,b)=(-1,\pm 1)$ (using PARI (v.2.9.1)), which gives $Z=W=1$. This shows that $(x,y,p,k,n)=(3^{r+1}, 3^{(2r+1)/3s}, 3,r,3s)$, $r,s\in\mathbb{Z}_{\geq 0}$ are also solutions of \eqref{me2}.  

 The second possibility is  
 $nt= \mbox{min} \{2s,2k+1,n t\}$ and in this situation \eqref{p20} becomes
 \begin{equation}\label{p21}
 p^{2s - n t}X^{2}+p^{2k- n t+1}=4Y^{n}.
 \end{equation}
This has solutions only if either $2 k = n  t-1$ or $2 s=n t$. Now if $2k=nt-1$, then \eqref{p21} gives 
 $$
 p^{2 s - n t}X^{2}+1=4Y^{n}.
 $$
Therefore as in the previous case, it has no solution.
On the other hand if $2 s = n t$, then \eqref{p21} implies
 \begin{equation}\label{p22}
  X^2+p^{2(k-s)+1}=4Y^{n}.
 \end{equation}
This has no solution by Proposition \ref{prop3.1} and Case I except for $n=2, 3,5,7$ and $13$. We deal with these values of $n$ individually.

If $n=2$ then $s=t$, and hence by Case I, the solutions of \eqref{p22} are given by 
$$(X, Y)=\left(\frac{p^{2(k-t)+1}-1}{2},~ \frac{p^{2(k-t)+1}+1}{4}\right).$$
Thus we can conclude that $$(x,y,k,2)=\left(p^t\times \frac{p^{2(k-t)+1}-1}{2},~ p^t \times \frac{p^{2(k-t)+1}+1}{4}, k, 2\right),\hspace*{2mm} t\in \mathbb{Z}_{\geq 0}$$ is a solution of \eqref{me2}.

Again if $n=3$ then $2s=3t$, and thus we can write $s=3r$ for some non-negative integer $r$. Therefore by Proposition \ref{prop3.1} the solutions of \eqref{p22} are given by 
$$\left(X, Y,p, n\right)\in\{(37,7, 3, 3),(5,2,7,3)\}$$
 satisfying $k=3r, ~r\in\mathbb{Z}_{\geq 0}$. 
This shows that 
$$(x,y,p, k, n)\in\{(37\times 3^{3r}, 7\times 3^{2r}, 3, 3r, 3),~(5\times 7^{3r}, 2\times 7^{2r}, 7, 3r, 3)\}, \hspace*{2mm} r\in \mathbb{Z}_{\geq 0},$$ are also solutions of \eqref{me2}.

For $n=5$, we have $2s=5t$ and thus we can write $s=5r$ for some $r\in\mathbb{Z}_{\geq0}$. Hence by Proposition \ref{prop3.1} the solutions of \eqref{p22} are 
given by $$(X, Y, p, n)\in \{(11, 2, 7, 5), ~ (31, 3, 11, 5)\}$$ under the condition that $k=5r, ~ r\in \mathbb{Z}_{\geq 0}$. 
This shows that 
$$(x,y,p, k, n)\in\{(11\times 7^{5r}, 2\times 7^{2r}, 7, 5r, 5),~(31\times 11^{5r}, 3\times 11^{2r}, 11, 5r, 5)\},$$
where $r\in \mathbb{Z}_{\geq 0}$, are also solutions of \eqref{me2}. 

Again for $n=7$, we have $2s=7t$ and thus we can write $s=7r$ for some $r\in\mathbb{Z}_{\geq 0}$. Therefore by Proposition \ref{prop3.1} the only solution of \eqref{p22} is given by 
$$\left(X, Y,p, n\right)=(13,2, 7, 7)$$
when $k=7r+1$ for some $r\in\mathbb{Z}_{\geq 0}$. 
This shows that 
$$(x,y,p, k, n)=(13\times 7^{7r}, 2\times 7^{2r}, 7, 7r+1, 7), \hspace*{2mm} r\in \mathbb{Z}_{\geq 0}$$ is also a solution of \eqref{me2}.

 Finally if $n=13$ then $2s=13t$, and thus we can write $s=13r$ for some non-negative integer $r$. Therefore by Proposition \ref{prop3.1} the only solution of \eqref{p22} is given by 
$$\left(X, Y,p, n\right)=(181,2, 7, 13)$$
satisfying the condition $k=3r, ~r\in\mathbb{Z}_{\geq 0}$. This further implies that  
$$(x,y,p, k, n)=(181\times 13^{13r}, 2\times 13^{2r}, 7, 13r, 13), \hspace*{2mm} r\in \mathbb{Z}_{\geq 0}$$ is also a solution of \eqref{me2}.

The only remaining case is $2s = \mbox{min} \{2s,2k+1,tn\}$. In this case, \eqref{p20} becomes 
 $$
 X^2+p^{2(k-s)+1}=4p^{tn-2s}Y^n.
 $$
Reading this modulo $p$,  we see that $t n = 2 s$, and thus it becomes
 $$
 X^2+p^{2(k-s)+1}=4Y^n.
 $$
This is same as \eqref{p22}. This completes the proof of Theorem \ref{thm}.  

\section*{acknowledgement}
The authors are thankful to Prof. Michel Waldschmidt for fruitful discussion while working on this manuscript. The second author is grateful to Prof. Kotyada Srinivas for stimulating environment at The Institute of Mathematical Sciences during his visiting period. He acknowledges SERB MATRICS Project No. MTR/2017/00100 and SERB-NPDF (PDF/2017/001958), Govt. of India. The third author  is thankful to Harish-Chandra Research Institute (HRI) and  Malaviya National Institute of Technology, Jaipur for providing facilities to prepare this manuscript. The authors are grateful to the anonymous referees for drawing their attention to \cite{AM01} and for their valuable remarks. The authors are also grateful to Prof. F. S. Abu Muriefah for providing a reprint of \cite{AM01}.

\end{document}